\documentclass{article}
\usepackage[T2A]{fontenc}
\usepackage[cp1251]{inputenc}
\usepackage{amssymb,amsmath,amsfonts,amscd,latexsym,theorem}
\usepackage[english,russian]{babel}
\usepackage{enumerate}
\usepackage[hyper]{amsbib}

\theoremstyle{plain}
\newtheorem{theorem}{Theorem}
\newtheorem{lemma}{Lemma}
\newtheorem{propos}{Proposition}
\theoremstyle{definition}

\newtheorem{proof}{Proof}

\begin{document}

\begin{center}
\textbf{ERGODICITY VIA CONTINUITY}

\medskip

\textbf{I.V. Podvigin}\footnote{Sobolev Institute of Mathematics and
Novosibirsk State University; email:ipodvigin@math.nsc.ru}

\end{center}


\renewcommand{\abstractname}{Abstract}
\begin{abstract} We show that the ergodicity of an aperiodic automorphism of a Lebesgue
space is equivalent to the continuity of a certain map on a metric
Boolean algebra. A related characterization is also presented for
periodic and totally ergodic transformations.
\end{abstract}

\text{}


\section{Introduction}

The main goal of this short note is to show that the ergodicity of
an aperiodic transformation $T$ of a Lebesgue space ${(\Omega,
\mathfrak{F}, \mu)}$ is equivalent to the continuity of a certain
transformation associated with $T$. There are many different but
equivalent definitions of ergodicity for measure preserving
transformations in the literature (see~\cite[\S2.3]{EW2011} for
example). Yet another criterion presented here seems to be new and
quite interesting.

Let ${(\mathcal{F},d)}$ be the metric space of $\mu$-equivalent
classes of $\mathfrak{F}$-measurable sets (a set
${A\in\mathfrak{F}}$ belongs to the class $[B]$ induced by a set
${B\in\mathfrak{F}}$ iff ${\mu(A\triangle B)=0}$). The metric~$d$ is
the Frechet--Nikodym metric defined as
$$
d([A],[B])=\mu(A\triangle B).
$$
Let ${\mathcal{N}\in\mathcal{F}}$ denote the class of sets of
$\mu$-measure zero. Given an automorphism $T,$ for each
$m\in\mathbb{N}$ define the map ${\phi^{(m)}_T:\mathcal{F}\to[0,1]}$
as
$$
\phi^{(m)}_T([A])=\mu\Bigl(\bigcup_{n=0}^mT^nA\Bigr).
$$
We also put
$$
\phi_T([A])=\lim_{m\to\infty}\phi^{(m)}_T([A])=\mu\Bigl(\bigcup_{n\geq0}T^nA\Bigr).
$$
The sequence ${\{\phi^{(m)}_T([A])\}_{m\in\mathbb{N}}}$ is called
the wandering rate of~$A$~\cite[\S3.8]{Aa97}.

It is well known that $T$ is ergodic iff ${\phi_T([A])=1}$ for every
${A\not\in\mathcal{N}}.$  It turns out equivalent to $\phi_T$ being
continuous everywhere except the point~$\mathcal{N},$ which is the
main statement (Theorem~\ref{Tm1}) in this note. We also present
related characterizations for periodic (Theorem~\ref{TmPeriodic})
and for totally ergodic transformations (Theorem~\ref{Tm2}).

\section{Continuity of the maps}

\noindent{\bf\thesection.1. Periodic transformation.} The first
local aim is to investigate the continuity of the transformations
${\phi^{(m)}_T}$, for $m\in\mathbb{N}.$ It is required for studying
the continuity of~${\phi_T}$ for periodic transformation~$T.$ The
continuity of ${\phi^{(m)}_T}$ is quite easy to prove by using the
methods of university measure theory courses. We give a proof of
this assertion for completeness of exposition.

\begin{lemma}\label{lemma1}  The
transformation ${\phi^{(m)}_T:\mathcal{F}\to[0,1]}$ is everywhere
continuous for each ${m\in\mathbb{N}}.$
\end{lemma}

\begin{proof}
It is evident that
$$
\Bigl(\bigcup_{n=0}^mT^nA\Bigr)\triangle\Bigl(\bigcup_{n=0}^mT^nB\Bigr)\subseteq\bigcup_{n=0}^m(T^nA\triangle
T^nB),
$$
and then
\begin{align*}
|\phi^{(m)}_T([A])&-\phi^{(m)}_T([B])|\leq\mu\Bigl(\Bigl(\bigcup_{n=0}^mT^nA\Bigr)\triangle\Bigl(\bigcup_{n=0}^mT^nB\Bigr)\Bigr)\leq\\
&\leq\mu\Bigl(\bigcup_{n=0}^m(T^nA\triangle
T^nB)\Bigr)\leq\sum_{k=0}^{m}\mu(T^n(A\triangle
B))=(m+1)\mu(A\triangle B).
\end{align*}
This completes the proof. It is worth noting that even the Lipschitz
property of $\phi^{(m)}_T$ follows from the proof.
\end{proof}

Recall the definitions of periodic and aperiodic transformations. A
point ${x\in\Omega}$ is called periodic for $T$ if there exists a
number ${n\in\mathbb{N}}$ with ${T^nx=x},$ and the smallest of these
numbers is called the period of~$x.$ Denote the set of periodic
points of period $n\in\mathbb{N}$ by $P_n$ and the set of aperiodic
points by $P_0$. It is clear that
$$
\Omega=\bigsqcup_{n\geq0}P_n.
$$
If ${\mu(P_0)=0}$ then the automorphism $T$ is called almost
everywhere periodic (or shortly periodic). If ${\mu(P_0)=1}$ then
$T$ is an aperiodic transformation.

The following proposition on the continuity of $\phi_T$ for periodic
automorphisms~$T$ is a corollary of Lemma~\ref{lemma1}.

\begin{propos}\label{ProposPeriod}
Let $T$ be a periodic automorphism of a probability space ${(\Omega,
\mathfrak{F}, \mu)}.$ Then the transformation $\phi_T$ is everywhere
continuous.
\end{propos}

\begin{proof}  Since
${\sum_{n=1}^\infty\mu(P_n)=1},$ for arbitrary~${\varepsilon>0}$
there exists a number~${n_0=n_0(\varepsilon)\geq1}$ such that
$$
\sum_{n=n_0+1}^\infty\mu(P_n)<\varepsilon/2.
$$
For every~${A\in\mathfrak{F}},$ put
$$
A_{n_0}=A\bigcap\Bigl(\bigcup_{n=1}^{n_0}P_n\Bigr),
$$
and then
$$
\phi_T([A])=\phi_T([A_{n_0}])+\phi_T([A\setminus
A_{n_0}])=\phi^{(n_0!-1)}_T([A_{n_0}])+\phi_T([A\setminus A_{n_0}]).
$$
As soon as
$$
d([A],[B])=\mu(A_{n_0}\triangle B_{n_0})+\mu((A\setminus
A_{n_0}))\triangle(B\setminus B_{n_0})<\varepsilon/2n_0!,
$$
the last calculation in the proof of Lemma~\ref{lemma1} yields
\begin{align*}
|\phi_T([A])-\phi_T([B])|&\leq|\phi^{(n_0!-1)}_T([A_{n_0}])-\phi^{(n_0!-1)}_T([B_{n_0}])|+\\
&+\mu\Bigl(\bigcup_{n\geq0}T^n(A\setminus
A_{n_0})\triangle\bigcup_{n\geq0}T^n(B\setminus B_{n_0})\Bigr)\leq\\
&<\varepsilon/2+\mu\Bigl(\bigcup_{n=n_0+1}^\infty
P_n\Bigr)<\varepsilon/2+\varepsilon/2=\varepsilon.
\end{align*}
This completes the proof.
\end{proof}

The converse to Proposition~\ref{ProposPeriod} is discussed in the
next subsection.

\noindent{\bf\thesection.2. The points of continuity.} The following
statements describe in detail all points of continuity of $\phi_T.$

\begin{lemma}\label{Propos2}
Let $T$ be an automorphism of a Lebesgue space ${(\Omega,
\mathfrak{F}, \mu)}.$ If ${\phi_T([A])=1}$ then $[A]$ is a
continuity point of $\phi_T.$
\end{lemma}

\begin{propos}\label{Propos2+}
Suppose that $T$ is an automorphism of a Lebesgue space ${(\Omega,
\mathfrak{F}, \mu)}$ and ${\mu(P_0)>0}.$ Then $\mathcal{N}$ is a
discontinuity point of~$\phi_T.$ Moreover, if $T$ is aperiodic then
${\phi_T([A])<1}$ iff $[A]$ is a discontinuity point of~$\phi_T.$
\end{propos}

Before proving these assertions, we remark that
Propositions~\ref{ProposPeriod} and~\ref{Propos2+} together imply
the following characterization of periodic automorphisms of a
Lebesgue space.
\begin{theorem}\label{TmPeriodic}
An automorphism $T$ of a Lebesgue space ${(\Omega, \mathfrak{F},
\mu)}$ is periodic iff $\phi_T$ is everywhere continuous.
\end{theorem}

\begin{proof}[of Lemma~\ref{Propos2}]
Consider the partition
$$
\Omega=\bigcup_{\alpha\in I}\Omega_\alpha
$$
of~$\Omega$ into ergodic components~${\Omega_\alpha}$ where $I$ is
the set of indices (see~\cite{Rud} for example). Express the
condition ~${\phi_T([A])=1}$ in terms of this ergodic decomposition:
$$
A=\bigcup_{\alpha\in I}A_\alpha,\ \ A_\alpha=A\cap\Omega_\alpha,
$$
and
$$
\mu\Bigl(\bigcup_{\alpha\in
I}\Omega_\alpha\Bigr)=1=\phi_T([A])=\mu\Bigl(\bigcup_{n\geq0}T^nA\Bigr)=\mu\Bigl(\bigcup_{n\geq0}T^n\bigcup_{\alpha\in
I}A_\alpha\Bigr)=\mu\Bigl(\bigcup_{\alpha\in
I}\bigcup_{n\geq0}T^nA_\alpha\Bigr).
$$
It follows that
$$
0=\mu\Bigl(\Bigl(\bigcup_{\alpha\in
I}\Omega_\alpha\Bigr)\setminus\bigcup_{\alpha\in
I}\Bigl(\bigcup_{n\geq0}T^nA_\alpha\Bigr)\Bigr)=\mu\Bigl(\bigcup_{\alpha\in
I}\Bigl(\Omega_\alpha\setminus\Bigl(\bigcup_{n\geq0}T^nA_\alpha\Bigr)\Bigr)\Bigr).
$$
Consequently, for each~${J\subseteq I},$ we have
$$
\mu\Bigl(\bigcup_{\alpha\in
J}\Omega_\alpha\setminus\Bigl(\bigcup_{n\geq0}T^nA_\alpha\Bigr)\Bigr)=0.
$$
It is equivalent to
\begin{equation}\label{eq:AbsCont}
\mu\Bigl(\bigcup_{\alpha\in
J}\Omega_\alpha\Bigr)=\mu\Bigl(\bigcup_{\alpha\in
J}\bigcup_{n\geq0}T^nA_\alpha\Bigr).
\end{equation}
On the set $I$ of indices consider the family of
measures~${\{\nu_C\}}_{C\in\mathfrak{F}}$ defined as
$$
\nu_C(J)=\mu\Bigl(\bigcup_{\alpha\in J}C_\alpha\Bigr),\ \ J\subseteq
I.
$$
We claim that~\eqref{eq:AbsCont} guarantees the equivalence of the
probability measure~${\nu_\Omega}$ and the measure~${\nu_A}.$ It is
clear that ${\nu_A\ll\nu_\Omega}.$ Suppose that the opposite is
false. Then there exists a set~${J\subset I}$ with
$$
\nu_A(J)=0, \ \ \text{but}\ \ \ \nu_\Omega(J)>0.
$$
It follows that
\begin{align*}
0=\mu\Bigl(\bigcup_{n\geq0}T^n\Bigl(\bigcup_{\alpha\in
J}A_\alpha\Bigr)\Bigr)=\mu\Bigl(\bigcup_{\alpha\in
J}\Bigl(\bigcup_{n\geq0}T^nA_\alpha\Bigr)\Bigr)=\mu\Bigl(\bigcup_{\alpha\in
J}\Omega_\alpha\Bigr)=\nu_\Omega(J)>0,
\end{align*}
which is a contradiction. Consequently, for each~${\varepsilon>0}$
there exists ${\delta=\delta(A,\varepsilon)>0}$ such that
${\nu_A(J)<\delta}$ implies ${\nu_\Omega(J)<\varepsilon}.$

Now we are ready to prove the continuity of $\phi_T$ at the point
$[A].$ For arbitrary $\varepsilon>0,$ take $\delta>0$ as in the
previous discussion. For ${B\in\mathfrak{F}}$ with ${\mu(B)>0},$
because the set ${\bigcup_{n\in\mathbb{Z}}T^nB}$ is invariant
under~$T,$ there exists a set ${J=J(B)\subset I}$ of indices such
that
$$
\bigcup_{\alpha\in J}\Omega_\alpha=\bigcup_{n\in\mathbb{Z}}T^nB\ \ \
(\mathrm{mod}\ \mu).
$$
This yields
$$
\mu\Bigl(\bigcup_{\alpha\in
J}\Omega_\alpha\Bigr)=\mu\Bigl(\bigcup_{n\in\mathbb{Z}}T^nB\Bigr)=\mu\Bigl(\bigcup_{n\geq0}T^nB\Bigr)
$$
and
$$
\mu\Bigl(\bigcup_{\alpha\in I\setminus J}B_\alpha\Bigr)=0.
$$
Assume that ${d([A],[B])<\delta}.$ We have
\begin{align*}
d([A],[B])&=\mu(A\triangle B)=\mu\Bigl(\bigcup_{\alpha\in
J}(A_\alpha\triangle B_\alpha)\Bigr)+\mu\Bigl(\bigcup_{\alpha\in
I\setminus J}(A_\alpha\triangle
B_\alpha)\Bigr)=\\
&=\mu\Bigl(\bigcup_{\alpha\in I\setminus
J}A_\alpha\Bigr)+\mu\Bigl(\bigcup_{\alpha\in J}(A_\alpha\triangle
B_\alpha)\Bigr)=\nu_A(I\setminus J)+\mu\Bigl(\bigcup_{\alpha\in
J}(A_\alpha\triangle B_\alpha)\Bigr)<\delta,
\end{align*}
which implies ${\nu_A(I\setminus J)<\delta}$ and, hence,
${\nu_\Omega(I\setminus J)<\varepsilon}.$ It follows that
\begin{align*}
\phi_T([A])-\phi_T([B])&=1-\mu\Bigl(\bigcup_{n\geq0}T^nB\Bigr)=\\
&=\mu\Bigl(\bigcup_{\alpha\in
I}\Omega_\alpha\Bigr)-\mu\Bigl(\bigcup_{\alpha\in J}\Omega_\alpha\Bigr)=\\
&=\nu_\Omega(I\setminus J)<\varepsilon,
\end{align*}
completing the proof.
\end{proof}

\begin{proof}[of Proposition~\ref{Propos2+}]
Without loss of generality, assume that the automorphism $T$ is
aperiodic.  Suppose that ${\phi_T([A])<1},$ and then the set
${B=\bigcup\limits_{n\in\mathbb{Z}}T^nA}$ of measure ${\mu(B)<1}$ is
$T$-invariant.

The restriction of~$T$ to ${\Omega\setminus B}$ is aperiodic and
preserves  the probability measure~$\mu_{\Omega\setminus B.}$
Applying the Rokhlin--Halmos lemma (see~\cite{KSF} for example), we
find that for~${\varepsilon>0}$ and $n_0\geq1$ there exists a set
${E\subset \Omega\setminus B}$ such that the sets ${T^kE}$ for
${0\leq k\leq n_0-1}$ are disjoint and satisfy the inequality
$$
{\mu_{\Omega\setminus
B}\Bigl(\bigcup_{k=0}^{n_0-1}T^kE\Bigr)>1-\varepsilon}.
$$
It is clear that $\mu_{\Omega\setminus B}(E)<\frac{1}{n_0}.$ Put
${C=A\cup E}.$ Then
$$
d([A],[C])=\mu(A\triangle C)=\mu(E)<\frac{1}{n_0}\mu(\Omega\setminus
B)
$$
and
\begin{align*}
\phi_T([C])&=\mu\Bigl(\bigcup_{n\geq0}T^nC\Bigr)=\mu\Bigl(\bigcup_{n\geq0}T^nA\Bigr)+\mu\Bigl(\bigcup_{n\geq0}T^nE\Bigr)\geq\\
&\geq\phi_T([A])+\mu\Bigl(\bigcup_{k=0}^{n_0-1}T^kE\Bigr)>\phi_T([A])+(1-\varepsilon)\mu(\Omega\setminus B).\\
\end{align*}
In this way, taking~${\varepsilon=1/2}$ and sufficiently large
${n_0\geq1}$ we obtain $d([A],[C])$ is small enough but
\begin{equation}\label{eq:discont}
{|\phi_T([A])-\phi_T([C])|>\mu(\Omega\setminus B)/2}.
\end{equation}
This proves that $\phi_T$ is discontinuous at the point~${[A]}$ with
${\phi_T([A])<1}.$

Now, if $[A]$ is a discontinuity point of ${\phi_T([A])}$ then
Lemma~\ref{Propos2} tells us that ${\phi_T([A])<1}.$ The proof is
complete.
\end{proof}


\noindent{\bf\thesection.3. Ergodic and totally ergodic
transformations.} The following characterization of ergodic
transformations is also direct corollary of
Proposition~\ref{Propos2+}.

\begin{theorem}\label{Tm1}
An aperiodic automorphism $T$ of a Lebesgue space
${(\Omega,\mathfrak{F},\mu)}$ is ergodic iff
${\phi_T:\mathcal{F}\to[0,1]}$ is continuous everywhere except the
point $\mathcal{N}$.
\end{theorem}

\begin{proof}
If $T$ is ergodic then ${\phi_T([A])=1}$ for ${\mu(A)>0}.$
Consequently, Lemma~\ref{Propos2} shows that $\phi_T$ is continuous
at such points. However, Proposition~\ref{Propos2+} states that
$\mathcal{N}$ is a discontinuity point.

Now, if $T$ is not ergodic then there exists an
$\mathfrak{F}$-measurable $T$-invariant set $A$ with ${0<\mu(A)<1}$.
It follows that ${\phi_T([A])=\mu(A)<1}$ and
Proposition~\ref{Propos2+} implies that $\phi_T$ has a discontinuity
at $[A].$
\end{proof}

As another application of Proposition~\ref{Propos2+}, we discuss
here a characterization of totally ergodic transformations, which
means that the powers $T^n$ for all $n\in\mathbb{N}$ are ergodic
transformations.

Define a new map ${\phi_T^*:\mathcal{F}\to[0,1]}$ as
$$
\phi_T^*([A])=\inf_{m\in\mathbb{N}}\phi_{T^m}([A]).
$$

\begin{theorem}\label{Tm2}
An aperiodic automorphism $T$ of a Lebesgue space
${(\Omega,\mathfrak{F},\mu)}$ is totally ergodic iff
${\phi^*_T:\mathcal{F}\to[0,1]}$ is continuous everywhere except the
point~$\mathcal{N}$.
\end{theorem}

\begin{proof}
If $T$ is totally ergodic then $\phi_{T^k}([A])=1$ for all $k\geq1$
and all ${A\not\in\mathcal{N}}.$ Hence, ${\phi^*([A])=1},$ and
therefore $\phi^*$ is continuous at that point. Indeed, for
arbitrary ${\varepsilon>0}$ we take ${0<\delta<\mu(A)}.$ Then the
inequality ${d([A],[B])<\delta}$ implies that ${\mu(B)>0}$ and
therefore
$$
|\phi^*_T([A])-\phi^*_T(B)|=|1-1|=0<\varepsilon.
$$
It is evident, that the class~$\mathcal{N}$ is a discontinuity point
of~${\phi^*_T}.$

Assume now that $T$ is not totally ergodic. Take the
smallest~${k_0\geq1}$ such that $T^{k_0}$ is not ergodic. Therefore
all powers $T^{nk_0}$ for ${n\geq 1}$ are not ergodic either
(because the invariant sets of $T^{k_0}$ are invariant under
$T^{nk_0}$ for all ${n\geq 1}$). It is clear that there are only
finitely many, at most $k_0,$ such sequences of non ergodic
transformations~${\{T^{nk}\}_{n\geq 1}}.$  Denote by~$\mathcal{K}$
the finite set of possible values of $k.$ Thus,
${k_0\in\mathcal{K}}$ and ${|\mathcal{K}|\leq k_0}.$ Put
${\kappa=\prod_{k\in\mathcal{K}}k}.$ Take  a nontrivial invariant
set ${A\not\in\mathcal{N}}$ of the transformation~${T^\kappa}.$ We
claim that it is a discontinuity point of~$\phi^*_T.$

By Theorem~\ref{Tm1}, all transformations~$\phi_{T^{n\kappa}}$ for
${n\geq1}$ (being non ergodic) are disconti\-nuous at $[A].$
Considering~\eqref{eq:discont}, we conclude that for sufficiently
small~$\delta>0$ and all ${n\geq1}$ there exist some sets
${E_n\subset\Omega\setminus A}$ such that ${B_n=A\cup E_n}$ satisfy
$$
d([A],[B_n])=\mu(E_n)<\frac{\delta}{2^n}\ \ \ \text{and}\ \ \
\phi_{T^{n\kappa}}([B_n])-\phi_{T^{n\kappa}}([A])>\frac{1}{2}\mu(\Omega\setminus
A).
$$
For the set ${B=A\bigcup(\bigcup_{n\geq1}E_n)}$ it is easy to see
that ${d([A],[B])=\mu(\bigcup_{n\geq1}E_n)<\delta},$ and for all
${n\geq 1},$
\begin{equation}\label{eq:estimate}
\phi_{T^{n\kappa}}([B])-\phi_{T^{n\kappa}}([A])\geq\phi_{T^{n\kappa}}([B_n])-\phi_{T^{n\kappa}}([A])>\frac{1}{2}\mu(\Omega\setminus
A).
\end{equation}
Considering the value~${\phi^*_T([B])},$ we conclude that
$$
\phi^*_T([B])=\inf_{m\in\mathbb{N}}\phi_{T^m}([B])=\min_{k\in\mathcal{K}}\inf_{n\in\mathbb{N}}\phi_{T^{kn}}([B])=\inf_{n\in\mathbb{N}}\phi_{T^{k'n}}([B])
$$
for some~${k'=k'(B)\in\mathcal{K}}.$ The second equality is true
because for the ergodic transfor\-mation~$T^m$ we have
${\phi_{T^m}([B])=1},$ and therefore the infimum is reached on non
ergodic transfor\-mations.

For arbitrary $\varepsilon>0$ there exists a
number~${n'\in\mathbb{N}}$ such that
\begin{equation}\label{eq:infim}
\phi^*_T([B])=\inf_{n\in\mathbb{N}}\phi_{T^{k'n}}([B])\geq\phi_{T^{k'n'}}([B])-\varepsilon.
\end{equation}
Taking into account the estimates~\eqref{eq:estimate} and
\eqref{eq:infim}, the monotonicity property
$$
\phi_{T^{k'n'}}\geq\phi_{T^{n'\kappa}},
$$
and the equality
$$
{\phi^*_T([A])=\phi_{T^{n'\kappa}}([A])=\mu(A)},
$$
we obtain
\begin{align*}
\phi^*([B])-\phi^*([A])&\geq\phi_{T^{k'n'}}([B])-\varepsilon-\phi^*([A])\geq\\
&\geq\phi_{T^{n'\kappa}}([B])-\phi_{T^{n'\kappa}}([A])-\varepsilon>\frac{1}{2}\mu(\Omega\setminus
A)-\varepsilon.
\end{align*}
\end{proof}
Take sufficiently small $\varepsilon>0$ so that the expression
${\frac{1}{2}\mu(\Omega\setminus A)-\varepsilon}$ is positive. Then
the last estimate guarantees that $\phi^*_T$ is discontinuous
at~$[A].$ The proof is complete.

In conclusion, we remark that it would be interesting to find a
related characterization for transformations with a different type
of mixing property (see~\cite{KY07} for example).

ACKNOWLEDGMENTS.  The work was supported by the program of
fundamental scientific  research of SB RAS № I.1.2., project №
0314-2016-0005.

\renewcommand{\refname}{References}

\end{document}